\documentclass{amsart}
\usepackage{amsmath, ytableau, xcolor}

\newcommand{\Mod}[1]{\ (\mathrm{mod}\ #1)}
\newtheorem{Lemma}             {Lemma}
\newtheorem{Corollary}  [Lemma]{Corollary}
\newtheorem{Proposition}[Lemma]{Proposition}
\newtheorem{Theorem}    [Lemma]{Theorem}

\newcommand{\op}{\operatorname}
\newcommand{\sym}{\mathfrak S}
\newcommand{\alt}{\mathfrak A}

\newcommand{\interchange}{\searrow\hspace{-1em}\swarrow}

\definecolor{lightred}{HTML}{CC6050}
\definecolor{darkred}{HTML}{981030}
\definecolor{mediumgreen}{HTML}{DDCC77}
\definecolor{lightblue}{HTML}{80CCEE}
\definecolor{darkblue}{HTML}{234280}

\definecolor{carmine}{HTML}{660000}

\title{A bijection for Euler's partition theorem in the spirit of Bressoud}
\author{John Murray}
\address{Department of Mathematics \& Statistics, Maynooth University, Co. Kildare, Ireland}
\email{John.Murray@mu.ie}
\date{March 29, 2018}
\begin{document}
\begin{abstract}
For each positive integer $n$, we construct a bijection between the odd partitions and the distinct partitions of $n$ which extends Bressoud's bijection between the odd-and-distinct partitions of $n$ and the splitting partitions of $n$.

We compare our bijection with the classical bijections of Glaisher and Sylvester, and also with a recent bijection due to Chen, Gao, Ji and Li.
\end{abstract}
\maketitle

\section{Introduction}

Euler's Theorem is that the number of odd partitions of a positive integer $n$ equals the number of distinct partitions of $n$. Many bijections have been constructed which realise this equality. The most well-known ones are due to Glaisher, Sylvester and Dyson (3.2, 3.4 and 3.5 in \cite{P}, respectively). Applications of Dyson's bijection have been explored in \cite{CJ}.

In this note we construct another bijection betwen the odd partitions of $n$ and the distinct partitions of $n$. Suppose that the odd partition $\lambda$ corresponds to the strict partition $\mu$ under our bijection. Then the number of parts of $\lambda$ equals the alternating sum of $\mu$ (as defined in \eqref{E:alt} below) and the number of parts occuring with odd multiplicity in $\lambda$ equals the number of odd parts in $\mu$. A different bijection with these properties was first constructed in \cite{CGJL}.

Our bijection extends a bijection of Bressoud \cite{Bd} between the odd and distinct partitions of $n$ and the splitting partitions of $n$ (as defined in \eqref{E:split} below).

\section{Preliminaries}

\subsection{Partitions}\label{SS:notation} A partition of $n\geq0$ is a non-increasing sequence $\lambda=(\lambda_1,\dots,\lambda_\ell)$ of positive integers whose sum is $n$. We express this by writing $\lambda\vdash n$. So $\ell=\ell(\lambda)$ is the number of parts in $\lambda$ and $n=|\lambda|$ is the sum of the parts of $\lambda$. As is customary we set $\lambda_j:=0$, for all $j>\ell$. The empty partition $()$ is the unique partition of $0$. The sets of odd and distinct partitions of $n$ are, respectively
$$
\begin{aligned}
&{{\mathcal O}(n):=\{\lambda\vdash n\mid\mbox{$\lambda_j$ is odd, for $j=1,\dots\ell$}\}},\\
&{{\mathcal D}(n):=\{\lambda\vdash n\mid\lambda_1>\lambda_2>\dots>\lambda_\ell\}}. 
\end{aligned}
$$

Define $\lambda_i':=\#\{j>0\mid\lambda_j\geq i\}$ for $i=1,\dots,\lambda_1$. Then $\lambda':=(\lambda_1',\lambda_2',\dots)$ is a partition of $n$ called the conjugate or transpose of $\lambda$. Notice that $\lambda_1'=\ell(\lambda)$, $\ell(\lambda')=\lambda_1$ and $(\lambda')'=\lambda$. The multiplicity of $i$ as a part of $\lambda$ is then $m_i(\lambda):=\lambda_i'-\lambda_{i+1}'$, for $i\geq1$. We set $n_o(\lambda):=\#\{i>0\mid\mbox{ $m_i(\lambda)$ is odd}\}$ as the number of parts which occur with odd multiplicity in $\lambda$. We also use the multiset notation for partitions. For example $(5,5,5,4,3,3,1,1,1,1)$ can be represented by $(5^3,4,3^2,1^4)$.

The {\em alternating sum} of $\lambda$ is the non-negative integer
\begin{equation}\label{E:alt}
|\lambda|_{a}:=\sum_{j=1}^\ell(-1)^{j-1}\lambda_j,\quad\mbox{notation as in \cite{KY}.} 
\end{equation}
We use $\ell_o(\lambda)$ to denote the number of odd parts in $\lambda$. So $\ell_o(\lambda)=|\lambda'|_a$. 

Given partitions $\alpha,\beta$, we use $\alpha\circ\beta$ to denote the partition of $|\alpha|+|\beta|$ obtained by arranging the disjoint union of the parts of $\alpha$ and $\beta$ in non-increasing order. We use $\alpha^2$ to denote $\alpha\circ\alpha=(\alpha_1,\alpha_1,\alpha_2,\alpha_2,\dots)$. 

We define a {\em run} in $\lambda$ to be a maximal sequence $2m-1,2m-3,\dots,2m-2k+1$ of consecutive odd integers, each of which occurs with odd multiplicity in $\lambda$. In particular $2m+1$ and $2m-2k-1$ each occur with even multiplicity in $\lambda$. We say that the run {\em odd} if $k$ is an odd integer and $k<m$. So $1$ is not a term in an odd run. We use $\ell_r(\lambda)$ to denote the number of odd runs in $\lambda$.

Example: $\lambda=(23,21,19^2,17,15^5,13,9^3,5,3,1)$ has 4 runs $(23,21)$, $(17,15,13)$, $(9)$ and $(5,3,1)$. Of these, the first and last are not odd. So $\ell_r(\lambda)=2$.

\subsection{Diagrams} Recall the notion of the Young diagram $[\lambda]$ of a partition $\lambda$, as for example described in \cite[2.1]{P}. Each part $k$ of $\lambda$ is represented by a row of $k$ boxes in $[\lambda]$. Row are justified to the left, and larger rows are drawn above smaller rows.

Now $\lambda$ can also be represented by its 2-modular diagram $[\lambda]_2$. In $[\lambda]_2$ each part $2k-1$ or $2k$ of $\lambda$ is represented by a row of $k$ boxes. The number $1$ or $2$ is written into the last box of the row, as the part is odd or even, and the number 2 is written into all other boxes in the row.

We will draw  other diagrams for partitions to illustrate various bijections.

\subsection{Spin regular and splitting partitions}\label{SS:double} A composition of $n$ is a sequence of positive integers whose sum is $n$. Following \cite{Bn} {\em the double} of $\mu\vdash n$ is obtained by replacing each odd part $\mu_i$ by the pair $(\frac{\mu_i+1}{2},\frac{\mu_i-1}{2})$ and each even part $\mu_i$ by the pair $(\frac{\mu_i}{2}+1,\frac{\mu_i}{2}-1)$. So $\mu^D_{2j-1}-\mu^D_{2j}=1$ or $2$ as $\mu_j$ is odd or even, respectively.

D. Benson calls $\mu$ {\em spin regular} if $\mu^D$ is a partition which has distinct parts. It is easy to see that the set of spin regular partitions of $n$ is
$$
\mathcal{SR}(n):=\left\{\mu\in{\mathcal D}(n)\left|\,
		  \begin{aligned}
                   &\mu_{i-1}-\mu_i\geq4\qquad\mbox{ and }\\
                   &\mu_{i-1}-\mu_i\geq6,\mbox{ if $\mu_{i-1},\mu_i$ even,}
                  \end{aligned}\right.
                  \mbox{for $i=2,\dots\ell(\mu)$.}
                   \right\}.
$$

The doubles of the spin regular partitions of $n$ constitute the set
$$
{\mathcal D}_{\leq2}(n):=\{\nu\in{\mathcal D}(n)\mid\nu_{2j-1}-\nu_{2j}=\mbox{$1$ or $2$, for $0<2j\leq\ell(\nu)+1$}\}.
$$
So doubling defines a bijection $\mathcal{SR}(n)\xrightarrow{\cong}{\mathcal D}_{\leq2}(n)$.

The set of splitting partitions of $n$ is
\begin{equation}\label{E:split}
{\mathcal S}(n):=\left\{\nu\vdash n\,\,\left|\,
			    \begin{aligned}
                             &\nu_{2j-1}-\nu_{2j}=1\mbox{ or }2,\\
                             &\nu_{2j-1}+\nu_{2j}\not\equiv2\Mod4,
                            \end{aligned}\right.
			    \quad\mbox{for $0<2j\leq\ell(\nu)+1$}
                 \right\}
\end{equation}
In particular ${\mathcal S}(n)\subseteq{\mathcal D}_{\leq2}(n)$. It is shown in \cite{Bn} that ${\mathcal S}(n)$ indexes the $2$-modular irreducible represeentations of the symmetric group $\sym_n$ which split on restriction to the alternating group $\alt_n$.

\subsection{Bijections of Glaisher and Sylvester} For comparison purposes we briefly describe the classical bijections ${\mathcal O}(n)\xrightarrow{\cong}{\mathcal D}(n)$ of Glaisher and Sylvester. %

Let $\lambda\in{\mathcal O}(n)$ correspond to $\nu\in{\mathcal D}(n)$ under Glaisher's bijection. Then for each $i\geq0$ and $j>0$, the part $2^ij$ occurs in $\nu$ if and only if $2^i$ occurs in the binary expansion of $m_j(\lambda)$. For example $\lambda=(7^{13},5^3,3^2,1^5)$ corresponds to $\nu=(56,28,10,7,6,5,4,1)$. We note that $\ell_o(\nu)=\op{n_o}(\lambda)$ but generally $|\nu|_a\ne\ell(\lambda)$ for Glaisher's bijection.

For Sylvester's bijection, again let $\lambda\in{\mathcal O}(n)$ and set $d:=\op{max}\{j\mid\lambda_j\geq 2j-1\}$. So $\lambda_d\geq 2d-1$ and $\lambda_{d+1}<2d+1$. If $\lambda_d>2d-1$ then $\lambda_{2d+1}'=d$ and we set $\epsilon=0$. If instead $\lambda_d=2d-1$ then $\lambda_{2d-1}'=d$ and we set $\epsilon=1$. Then the image $\mu\in{\mathcal D}(n)$ of $\lambda$ under Sylvester's bijection has $2d-\epsilon$ parts defined by
$$
\begin{array}{ll}
\mu_{2j-1}&=\frac{\lambda_j+1}{2}+\lambda_{2j-1}'-2j+1,\quad j=1,\dots,d\\
\mu_{2j}  &=\frac{\lambda_j-1}{2}+\lambda_{2j+1}'-2j+1,\quad j=1,\dots,d-\epsilon.
\end{array}
$$

\subsection{Graphical description of Sylvester's bijection} This description is essentially the same as that of \cite[Figure 12]{P}. Draw a Young diagram of $\lambda\in{\mathcal O}(n)$, but with the rows centered rather than justified to the left. Now split the diagram with a vertical line so that in each row there is one more box to the left than to the right (so $\lambda^D$ is the composition we get from reading the left and right parts of the diagram from top to bottom). Decompose the left and right sides into {\em hooks}, as indicated in Figure \ref{f:sylvester} below. The hooks alternate in length from left and right, forming a strict sequence. This is the Sylvester correspondent $\mu$ of $\lambda$.

Notice that there are $\ell(\lambda)$ more boxes on the left than on the right. It follows from this that $|\mu|_{a}=\ell(\lambda)$, as first noted in \cite[2.2(iii)]{Bt}. On the other hand for Sylvester's bijection $\ell_o(\mu)\ne\op{n_o}(\lambda)$ generally.

Example: Sylvester map the odd partition $\lambda=(13^2,11,5^2,3,1^2)$ to the distinct partition $\mu=(14,11,10,8,6,3)$, as indicated in Figure \ref{f:sylvester}.

\begin{figure}[!ht]
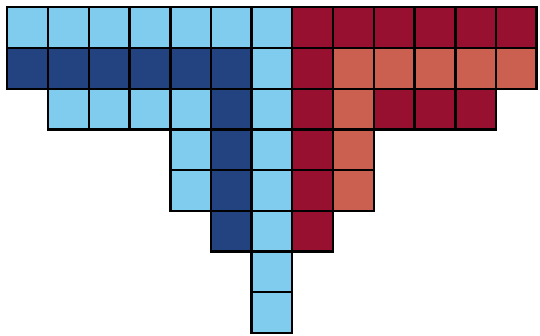

\caption{Sylvester's bijection for $(13^2,11,5^2,3,1^2)$}
\label{f:sylvester}
$$
\begin{ytableau}
*(lightblue) &*(lightblue) &*(lightblue) &*(lightblue) &*(lightblue) &*(lightblue) &*(lightblue) &*(darkred) &*(darkred) &*(darkred) &*(darkred) &*(darkred) &*(darkred)\\
*(darkblue) &*(darkblue) &*(darkblue) &*(darkblue) &*(darkblue) &*(darkblue) &*(lightblue) &*(darkred) &*(lightred) &*(lightred) &*(lightred) &*(lightred) &*(lightred)\\
\none &*(lightblue) &*(lightblue) &*(lightblue) &*(lightblue) &*(darkblue) &*(lightblue) &*(darkred) &*(lightred) &*(darkred) &*(darkred) &*(darkred)\\
\none &\none &\none &\none &*(lightblue) &*(darkblue) &*(lightblue) &*(darkred) &*(lightred)\\
\none &\none &\none &\none &*(lightblue) &*(darkblue) &*(lightblue) &*(darkred) &*(lightred)\\
\none &\none &\none &\none &\none &*(darkblue) &*(lightblue) &*(darkred)\\
\none &\none &\none &\none &\none &\none &*(lightblue)\\
\none &\none &\none &\none &\none &\none &*(lightblue)\\
\end{ytableau}
$$
\end{figure}

We note that Sylvester's bijection can also be described in terms of the 2-modular diagram $[\lambda]_2$ of $\lambda$. See \cite{Bd}.

\subsection{The inverse of Sylvester's bijection}\label{SS:Sylvester} We give a graphical construction for the inverse of Sylvester's map, which the reader can compare to Figure \ref{f:imy2-modular} below. This is only a slight modification of the method described by \cite[Figure 1]{KY}. Draw a diagram for $\mu\in{\mathcal D}(n)$ as follows. First arrange $\mu_1$ boxes horizontally to form the first row. The second row consisting of $\mu_2$ boxes is put below the first row, aligning the rightmost box in the second row with the {\em second to right} box in the row above. The third row consisting of $\mu_3$ boxes is put below the second row, aligning the leftmost box in the third row with the leftmost box in the row above. Repeat this process, alternating between left and right.

\begin{figure}[h]
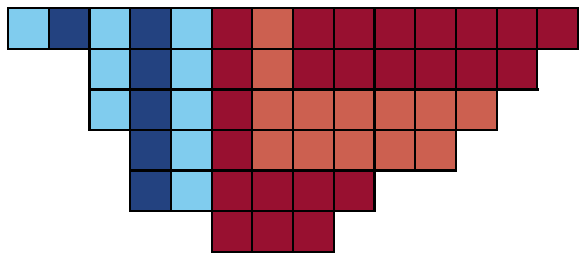

\caption{Inverse of Sylvester's bijection for $\mu=(14,11,10,8,6,3)$}
\label{f:isylvester}
$$
\begin{ytableau}
*(lightblue) &*(darkblue) &*(lightblue) &*(darkblue) &*(lightblue)&*(darkred) &*(lightred) &*(darkred) &*(darkred) &*(darkred) &*(darkred) &*(darkred) &*(darkred) &*(darkred) \\
\none &\none &*(lightblue) &*(darkblue) &*(lightblue) &*(darkred) &*(lightred) &*(darkred) &*(darkred) &*(darkred) &*(darkred) &*(darkred) &*(darkred)\\
\none &\none &*(lightblue) &*(darkblue) &*(lightblue) &*(darkred) &*(lightred) &*(lightred) &*(lightred) &*(lightred) &*(lightred) &*(lightred)\\
\none &\none &\none &*(darkblue) &*(lightblue) &*(darkred) &*(lightred) &*(lightred) &*(lightred) &*(lightred) &*(lightred)\\
\none &\none &\none &*(darkblue) &*(lightblue) &*(darkred) &*(darkred) &*(darkred) &*(darkred)\\
\none &\none &\none &\none &\none &*(darkred) &*(darkred) &*(darkred)\\
\end{ytableau}
$$
\end{figure}

Notice that all odd-length columns in the diagram are to the left. Their lengths encode the parts $1^2,3,5^2$ of $\lambda$. Removing these $5$ columns leaves the Young diagram of $(9,8,7,6,4,3)$. The boxes in the union of the first column and last two rows represent the part $11$ of $\lambda$. Removing these leaves the diagram of $(8,7,6,5)$. In this diagram the boxes in the union of the first column and last two rows represent the part $13$ of $\lambda$. Removing these leaves the diagram of $(7,6)$, which represents another part $13$ in $\lambda$. So our procedure has recovered all parts of $\lambda$.

\subsection{Bressoud's Bijection}\label{SS:Bressoud} In order to construct our bijection, we need a special case of a partition bijection discovered by D. Bressoud \cite{Bd}. He used it to give a combinatorial proof of a partition theorem of I. Schur \cite{S}. Schur's result states that if $r$ and $m$ are positive integers with $2r<m$, then the number of partitions of $n$ into distinct parts congruent to $\pm r\Mod m$ equals the number of partitions of $n$ into distinct parts congruent to $0,\pm r\Mod m$ with minimal difference $m$ and minimal difference $2m$ between multiples of $m$.

Taking the case $r=1$ and $m=4$ in Schur's Theorem, and forming the doubles of partitions, we see that the number of odd-and-distinct partitions of $n$ equals the number of splitting partitions of $n$. The following is implicit in \cite{Bd}:

\begin{Proposition}\label{P:Bressoud}
There is a bijection ${\mathcal O}(n)\cap{\mathcal D}(n)\xrightarrow{\cong}{\mathcal S}(n)$, $\lambda\mapsto\mu$ such that
$$
|\mu|_{a}=\ell_o(\mu)=\ell(\lambda)\quad\mbox{and}\quad\ell(\mu)=\ell(\lambda)+\ell_r(\lambda).
$$
\end{Proposition}

\begin{proof}
We obtain the map described in \cite{Bd}, but using a different algorithm. The first step is to construct a composition $\pi$ from $\lambda$ by a process of pairing and doubling, as described below. The parts $\pi_{2j-1},\pi_{2j}$ are called the pairs of $\pi$, and we demark pairs with vertical lines. We obtain $\mu$ from $\pi$ via a sequence of {\em pair interchanges}.

To construct $\pi$, let $2m-1,2m-3,\dots,2m-2k+1$ be a run in $\lambda$ (see Section \ref{SS:notation}). Pair off consecutive terms $(2m-1,2m-3)$, $(2m-5,2m-7),\dots$ proceeding from largest to smallest. If $k$ is even this exhausts all terms in the run. If $k$ is odd, this leaves the smallest term $2m-2k+1$ unpaired. In this event we double $2m-2k+1$, meaning that we replace it with the pair $(m-k+1,m-k)$. In particular a lonely $1$ becomes the pair $(1,0)$.

Let $\pi$ be the concatenation of all pairs, arranged as they are formed from left to right in $\lambda$. Notice that $|\pi|_{a}=\ell_o(\pi)=\ell(\lambda)$ as each pair $(2m+1,2m-1)$ in $\pi$ arises from 2 parts of $\lambda$ and contributes $2$ to $|\pi|_a$ and each pair $(m+1,m)$ in $\pi$ arises from 1 part of $\lambda$, contains 1 odd part and contributes 1 to $|\pi|_a$. Moreover $\ell(\pi)-\ell(\lambda)$ equals the number $\ell_r(\lambda)$ of odd runs in $\lambda$; a pair $(1,0)$ does not affect $\ell(\pi)-\ell(\lambda)$.

Suppose that $\pi$ is not a partition. Then it contains a subsequence of the form $(a+1,a\mid2b+1,2b-1)$ with $a\leq 2b+1$. Interchange the left and right pairs, and then add 2 to each part on the right and subtract 2 from each part on the left (notice that $a\geq 2$ as the right hand side represents two odd parts of $\lambda$ each less that $2a+1$). This gives a new subsequence $(2b+3,2b+1\mid a-1,a-2)$ which is in descending order and which has the same sum as the original subsequence. We call this operation a `pair interchange'.

Start with $\pi$ and sucessively apply pair interchanges as many times as possible. The resulting composition $\mu$ is a splitting partition of $n$. To conclude, notice that $|\mu|_{a}=\ell_o(\mu)=\ell(\lambda)$ and $\ell(\mu)=\ell(\lambda)+\ell_r(\lambda)$, because each pair interchange preserves the length, alternating length and number of odd parts.
\end{proof}

\medskip
{\bf Example:} consider the odd-distinct partition $\lambda=(29,21,19,17,13,11,7,5,1)$. Then $\lambda$ has runs $(29),(21,19,17),(13,11),(7,5)$ and $(1)$. So $\pi=(15,14\mid21,19\mid9,8\mid13,11\mid7,5\mid1,0)$. The sequence of pair exchanges which produces $\mu$ is
$$
\begin{array}{ccccccccccl}
(15,14&\mid&21,19&\mid&9,8&\mid&13,11&\mid&7,5&\mid&1,0)\\
&\interchange&&&&\interchange\\
(23,21&\mid&13,12&\mid&15,13&\mid&7,6&\mid&7,5&\mid&1,0)\\
&&&\interchange&&&&\interchange\\
(23,21&\mid&17,15&\mid&11,10&\mid&9,7&\mid&5,4&\mid&1,0)=\mu\\
\end{array}
$$

The method of \cite{Bd} in this example proceeds as follows: Form $\pi$ as above. Subtract $(10,10\,|\,8,8\,|\,6,6\,|\,4,4\,|\,2,2\,|\,0,0)$ from $\pi$, giving $(5,4\,|\,13,11\,|\,3,2\,|\,9,7\,|\,5,3\,|\,1,0)$. Rearrange pairs according to their sums to get $(13,11\,|\,9,7\,|\,5,4\,|\,5,3\,|\,3,2\,|\,1,0)$. Finally add back $(10,10\,|\,8,8\,|\,6,6\,|\,4,4\,|\,2,2\,|\,0,0)$ to obtain $\mu$.

The inverse of Bressoud's map uses reverse pair interchanges. Let $\mu\in{\mathcal S}(n)$ and let $t>0$ such that $\ell(\mu)\leq2t\leq\ell(\mu)+1$. Partition $\mu$ into pairs of sucessive parts $\mu=(\mu_1,\mu_2\mid\mu_3,\mu_4\mid\dots\mid\mu_{2t-1},\mu_{2t})$. Given a subsequence of parts $(2b+1,2b-1\mid a+1,a)$ with $2a+1>2b-1$, perform a {\em reverse pair interchange} to get $(a+3,a+2\mid 2b-1,2b-3)$. Repeat reverse pair interchanges as many times as is possible. This results in the a composition $\pi$. Then $\lambda$ is the unique partition whose paired composition is $\pi$.

\section{A new Bijection}\label{S:our_bijection}

In this section we prove the following theorem:

\begin{Theorem}\label{T:main}
Bressoud's bijection ${\mathcal O}(n)\cap{\mathcal D}(n)\xrightarrow{\cong}{\mathcal S}(n)$ can be extended to a bijection ${\mathcal O}(n)\xrightarrow{\cong}{\mathcal D}(n)$, $\lambda\mapsto\mu$ such that
$$
|\mu|_{a}=\ell(\lambda)\quad\mbox{and}\quad\ell_o(\mu)=n_o(\lambda).
$$
\end{Theorem}

\begin{proof}
Write $\lambda=\alpha^2\circ\beta$, for unique partitions $\alpha,\beta$ such that $\beta$ has distinct parts. So $\ell(\beta)=n_o(\lambda)$ and $\ell_r(\beta)=\ell_r(\lambda)$. Let $\gamma^{(0)}$ be the image of $\beta$ under Bressoud's map. 

Set $t:=\ell(\alpha)$. We construct a sequence of partitions $\gamma^{(1)},\dots,\gamma^{(t)}$, such that $\gamma^{(i)}$ is obtained by {\em inserting} $\alpha_i^2$ into $\gamma^{(i-1)}$, for $i=1,\dots,t-1$. Then each $\gamma^{(i)}$ has distinct parts and the final one $\gamma^{(t)}$ will be the image $\mu\in{\mathcal D}(n)$ of $\lambda$ under our bijection. Note that we only insert $\alpha_i^2$ after inserting the pairs $\alpha_1^2,\dots,\alpha_{i-1}^2$ as insertions do not commute.

In order to simplify notation we describe the insertion of an odd pair $(2a-1)^2$ into a partition $\gamma$. We cannot assume that $\gamma$ is a splitting partition, as $\gamma^{(i)}$ is not a splitting partition for $i>0$. However we can assume by induction that
\begin{equation}\label{E:gamma}
\begin{aligned}
&\gamma_{2j-1}-\gamma_{2j}=\mbox{$1$ or $2$ for $j<a$ and}\\
&\mbox{if $\gamma_{2j-1}+\gamma_{2j}\equiv2\Mod4$, with $1\leq j\leq a$, then $\gamma_{2j-1}\geq2(a+1-j)$.}
\end{aligned}
\end{equation}

Here is how pair insertion works. If $a=1$, replace $\gamma$ with the partition $(\gamma_1+2,\gamma_2\mid \gamma_3,\gamma_4\mid\dots)$. Otherwise represent $(2a-1)^2$ by the pair $(2a,2a-2)$ and form the composition $(2a,2a-2\mid\gamma_1,\gamma_2\mid\gamma_3,\gamma_4\mid\dots)$. If $2a-2>\gamma_1$, this composition is a partition, and we are done. Otherwise, perform a pair interchange on the first two pairs to get the composition $(\gamma_1+2,\gamma_2+2\mid 2a-2,2a-4\mid\gamma_3,\gamma_4\mid\dots)$. Notice that by our inductive hypothesis $\gamma_2+2\geq\gamma_1\geq2a$. So the first 4 terms in the new composition are in decreasing order.

Now if $2a-4>\gamma_3$, the new composition is a partition, and we are done. If instead $a=2$, replace the composition with $(\gamma_1+2,\gamma_2+2\mid \gamma_3+2,\gamma_4\mid\gamma_5,\dots)$. Otherwise, perform a pair interchange on the second and third pairs to get $(\gamma_1+2,\gamma_2+2\mid\gamma_3+2,\gamma_4+2\mid 2a-4,2a-6\mid\gamma_5,\gamma_6\mid\dots)$. Proceeding in this manner we obtain a partition of the form:
$$
\begin{aligned}
(\gamma_1+2,\gamma_2+2\mid\dots\mid\gamma_{2a-1}+2,\gamma_{2a}\mid\gamma_{2a+1},\gamma_{2j+2}\dots)\quad\mbox{if $2a-1\leq\ell(\gamma)+1$ or}\\ 
(\gamma_1+2,\gamma_2+2\mid\dots,\gamma_{2j}+2\mid2a-2j,2a-2j-2\mid\gamma_{2j+1},\gamma_{2j+2},\dots)\quad\mbox{otherwise.}\\ 
\end{aligned}
$$
It is clear that this partition satisfies the conditions \eqref{E:gamma} with respect to any odd pair $(2a'-1)^2$ with $a'\leq a$. Moreover the process of pair insertion shows that if $\gamma$  is a partition satisfying the conditions \eqref{E:gamma}, and $2a-1>\ell(\gamma)+1$ then
\begin{equation}\label{E:j}
\mbox{there is a unique $j\geq0$ such that $\gamma_{2j}>2a-2j-2>\gamma_{2j+1}$.}
\end{equation}

Now Proposition \ref{P:Bressoud} implies that $|\gamma^{(0)}|_a=\ell_o(\gamma^{(0)})=\ell(\beta)$. Moreover $|\gamma^{(i)}|_{a}=|\gamma^{(i-1)}|_{a}+2$ and $\ell_o(\gamma^{(i)})=\ell_o(\gamma^{(i-1)})$, for $i=1,\dots,t-1$. Also $\ell(\alpha^2)=2t$ and $\gamma^{(t)}=\mu$. We conclude that $|\mu|_{a}=\ell(\lambda)$ and $\ell_o(\mu)=n_o(\lambda)$. We prove that our map is invertible in Section \ref{SS:inverse}.\end{proof}

\subsection{Example} We apply the algorithm described in Theorem \ref{T:main} to the odd partition $\lambda=(13^4,11^2,9,5^5,3^3,1^4)$ of $121$. Then $\alpha=(13^2,11,5^2,3,1^2)$ and $\beta=(9,5,3)$. Bressoud's bijection maps $\beta$ to $\gamma^{(0)}=(7,5,3,2)$.

Now $\ell(\alpha)=8$. We construct $\gamma^{(1)},\dots,\gamma^{(8)}$ by sucessive insertions of the pairs $13^2,13^2,11^2,5^2,5^2,3^2,1^2,1^2$. To insert $13^2$ into $\gamma$, prepend $\gamma^{(0)}$ with $(14,12)$ to get $\gamma^{(1)}=(14,12\mid7,5\mid3,2)$. To insert the second $13^2$ into $\gamma^{(1)}$, first form $(14,12\mid14,12\mid7,5\mid3,2)$. Then one pair interchange produces $\gamma^{(2)}$:
$$
\begin{array}{cccccccc}
(\textcolor{darkred}{14,12}&\mid&\textcolor{darkblue}{14,12}&\mid&7,5&\mid&3,2)\\
&\interchange\\
(\textcolor{darkblue}{16,14}&\mid&\textcolor{darkred}{12,10}&\mid&7,5&\mid&3,2)&=\gamma^{(2)}\\
\end{array}
$$
Next $a_3=6$ and $11^2$ is initially represented by $(12,10)$. As $\gamma^{(1)}_6>2(6-3-1)>\gamma^{(1)}_7$, we have $j=3$ in \eqref{E:j}. So starting with $(\textcolor{darkred}{12,10}\mid{16,14}\mid{12,10}\mid{7,5}\mid3,2)$, and applying 3 pair interchanges, we get $\gamma^{(3)}=(\textcolor{darkblue}{18,16}\mid\textcolor{darkblue}{14,12}\mid\textcolor{darkblue}{9,7}\mid\textcolor{darkred}{6,4}\mid3,2)$.

Now for each remaining $\alpha_i=5,5,3,1,1$, we have $\alpha_i\leq\ell(\gamma^{(i-1)})+1$. So $\gamma^{(i)}$ is got by adding $2$ onto each of the first $\alpha_i$ parts of $\gamma^{(i-1)}$, for $i=4,5,6,7,8$. Thus
$$
\begin{array}{ccccccccccccccc}
\gamma^{(4)}&=&(\textcolor{darkred}{20},\textcolor{darkred}{18}&\mid&\textcolor{darkred}{16},\textcolor{darkred}{14}&\mid&\textcolor{darkred}{11},7&\mid&6,4&\mid&3,2)\\
\gamma^{(5)}&=&(\textcolor{darkred}{22},\textcolor{darkred}{20}&\mid&\textcolor{darkred}{18},\textcolor{darkred}{16}&\mid&\textcolor{darkred}{13},7&\mid&6,4&\mid&3,2)\\
\gamma^{(6)}&=&(\textcolor{darkred}{24},\textcolor{darkred}{22}&\mid&\textcolor{darkred}{20},16&\mid&13,7&\mid&6,4&\mid&3,2)\\
\gamma^{(7)}&=&(\textcolor{darkred}{26},22&\mid&20,16&\mid&13,7&\mid&6,4&\mid&3,2)\\
\gamma^{(8)}&=&(\textcolor{darkred}{28},22&\mid&20,16&\mid&13,7&\mid&6,4&\mid&3,2)\\
\end{array}
$$
Notice that adding the terms of $2(5,5,3,1,1)'=(10,6,6,4,4)$ to $\gamma^{(3)}$ produces $\gamma^{(8)}$ in one step.  

The calculation above shows that our bijection maps $\lambda$ to the distinct partition $\mu=(28,22,20,16,13,7,6,4,3,2)$. As expected $|\mu|_a=\ell(\lambda)$ and $\ell_o(\mu)=n_o(\lambda)$, with common values $19$ and $3$, respectively.

\subsection{The inverse}\label{SS:inverse} Let $\mu\in{\mathcal D}(n)$. We will show how to construct an odd partition $\lambda$ of $n$ such that $\mu$ is the image of $\lambda$ under the algorithm described in the proof of Theorem \ref{T:main}.

Let $t>0$ such that $\ell(\mu)\leq2t\leq\ell(\mu)+1$. Then for each $j=1,\dots,t$ there is a unique non-negative integer $m_j$ such that $\mu_{2j-1}-\mu_{2j}-2m_j=1$ or $2$. Subtracting the parts of the conjugate of $((2t-1)^{2m_t},(2t-3)^{m_{2t-1}},\dots,3^{2m_3},1^{2m_1})$ from $\mu$ leaves a partition $\gamma$. This reverses the process of inserting, for $j=t,t-1,\dots,1$ in turn, $m_j$ pairs $(2j-1)^2$ into $\gamma$ to produce $\mu$, in the language of the proof of Theorem \ref{T:main}. This contributes $m_j$ pairs $(2j-1)^2$ to $\lambda$, for each $j=1,\dots,t$. 

Next we work with $\gamma$. Now $\gamma_{2j-1}-\gamma_{2j}=1$ or $2$, for $j=1,\dots,t$. Suppose that $\gamma$ is not a splitting partition. Then there is a largest $j\geq0$ such that $\gamma_{2j+1}$ and $\gamma_{2j+2}$ are both even. Set $\gamma':=(\gamma_1-2,\dots,\gamma_{2j}-2,\gamma_{2j+3},\gamma_{2j+4},\dots,\gamma_{2t})$. By the condition in \eqref{E:j}, inserting the pair $(\gamma_{2j+1}+2j-1)^2$ into $\gamma'$ produces $\gamma$. This contributes a pair $(\gamma_{2j}+2j-1)^2$ to $\lambda$. Repeat this procedure with $\gamma'$, if necessary.

The process above terminates in a splitting partition $\gamma^{(0)}$. Let $\beta$ be the odd-distinct partition which corresponds to $\nu$ under the inverse of Bressoud's bijection. Taking the disjoint union of the parts of $\beta$ with all the parts extracted from $\mu$ above produces the odd partition $\lambda$.

\subsection{Graphical description of the inverse}\label{SS:our_inverse}  We give a graphical construction for the inverse of the bijection of Theorem \ref{T:main}. Let $\mu\in{\mathcal D}(n)$ and recall that $[\mu]_2$ is the 2-modular diagram of $\mu$. Let $r_i$ be the $i$-th row in $[\mu]_2$, for $i=1,\dots\ell(\mu)$. So $r_i$ consists of $\lfloor\frac{\mu_i+1}{2}\rfloor$ boxes; the rightmost box contains a $1$ or $2$ as $\mu_i$ is odd or even, respectively, and all other boxes contain a $2$.

Now $r_1$ forms the first row of our diagram. Place $r_2$ below the first row, aligning the rightmost box in $r_2$ with the {\em second to right} box in the first row. The third row $r_3$ is put below the second row, aligning the leftmost box in $r_3$ with the leftmost box in the second row. Repeat this process, alternating between left and right (cf. the diagram of Section \ref{SS:inverse} above).

For example, consider the distinct partition $\mu=(28,22,20,16,13,7,6,4,3,2)$, which is the image of the odd partition $\lambda=(13^4,11^2,9,5^5,3^3,1^4)$ under our bijection.

\begin{figure}[h]
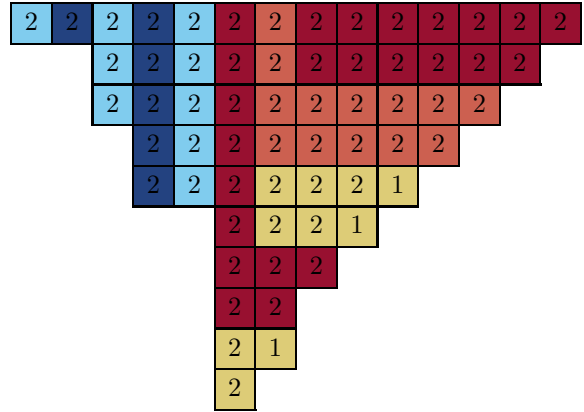

\caption{Inverse of our bijection for $\mu=(28,22,20,16,13,7,6,4,3,2)$}
\label{f:imy2-modular}
$$
\begin{ytableau}
*(lightblue)2 &*(darkblue)2 &*(lightblue)2 &*(darkblue)2
&*(lightblue)2 &*(darkred)2 &*(lightred)2 &*(darkred)2
&*(darkred)2 &*(darkred)2 &*(darkred)2 &*(darkred)2
&*(darkred)2 &*(darkred)2\\
\none&\none&*(lightblue)2 &*(darkblue)2 &*(lightblue)2 &*(darkred)2
&*(lightred)2 &*(darkred)2 &*(darkred)2 &*(darkred)2
&*(darkred)2 &*(darkred)2 &*(darkred)2 \\
\none&\none&*(lightblue)2 &*(darkblue)2 &*(lightblue)2 &*(darkred)2
&*(lightred)2 &*(lightred)2 &*(lightred)2 &*(lightred)2
&*(lightred)2 &*(lightred)2\\
\none&\none&\none&*(darkblue)2 &*(lightblue)2 &*(darkred)2 &*(lightred)2
&*(lightred)2 &*(lightred)2 &*(lightred)2 &*(lightred)2\\
\none&\none&\none&*(darkblue)2 &*(lightblue)2 &*(darkred)2 &*(mediumgreen)2
&*(mediumgreen)2 &*(mediumgreen)2 &*(mediumgreen)1\\
\none&\none&\none&\none&\none&*(darkred)2 &*(mediumgreen)2 &*(mediumgreen)2 &*(mediumgreen)1\\
\none&\none&\none&\none&\none&*(darkred)2 &*(darkred)2 &*(darkred)2\\
\none&\none&\none&\none&\none&*(darkred)2 &*(darkred)2\\
\none&\none&\none&\none&\none&*(mediumgreen)2 &*(mediumgreen)1\\
\none&\none&\none&\none&\none&*(mediumgreen)2\\
\end{ytableau}
$$
\end{figure}

The 4 leftmost columns in the diagram have odd length and represent the pairs of parts $(5^2,5^2,3^2,1^2,1^2)$ of $\lambda$. Removing these columns yields the 2-modular diagram of $\gamma=(18,16,14,12,9,7,6,4,3,2)$. By construction $\gamma_{2j-1}-\gamma_{2j}=1$ or $2$, for $j=1,\dots,5$.

Let $j$ be the largest positive integer such that the rightmost box in each of the rows $2j-1$ and $2j$ contains a $2$. Remove the boxes in rows $2j-1$ and $2j$ and all boxes which are in the first column and in one of the first $2j-2$ rows of $\gamma$. This yields the pair of odd parts $(\frac{\gamma_{2j-1}+\gamma_{2j}}{2}+2j-2)^2$ for $\lambda$. In our example $j=4$. So the sum of the contents of the removed boxes is $22$, yielding the pair of parts $11^2$ for $\lambda$. The remaining boxes form the 2-modular diagram of $(16,14,12,10,7,5,3,2)$.

Continue this process. The next iteration removes the remaining boxes in rows $3$ and $4$ and the $2$ boxes in rows $1$ and $2$ which are in the first column. The sum of the contents of these boxes is $26$. This encodes the pair of parts $13^2$ of $\lambda$. What remains is the 2-modular diagram of $(14,12,7,5,3,2)$. The final iteration yields $26$ from the contents of the remaining first two rows, encoding another pair of parts $13^2$ of $\lambda$.

Having exhausted these box removals, we are left with the 2-modular diagram of $\gamma=(7,5,3,2)\in{\mathcal S}(17)$. As we showed earlier, Bressoud maps this to $(9,5,3)\in{\mathcal O}(17)\cap{\mathcal D}(17)$. Concatenate this with the extracted pairs, and rearrange the parts into descending order in order to obtain $\lambda=(13^4,11^2,9,5^5,3^3,1^4)$.

\subsection{Comparison with the bijection of Chen, Gao, Ji and Li}

As we noted in the introduction, a bijection ${\mathcal O}(n)\rightarrow{\mathcal D}(n)$ which maps length to alternating sum and number of odd multiplicity parts to number of odd parts was first constructed in \cite{CGJL}. According to the example in that paper, the image of the distinct partition $(17,16,14,10,7,4,2,1)\in{\mathcal D}(71)$ under this bijection is the odd partition $(19,13,9^2,5^3,3^2)$. However, the reader can check that the inverse of the bijection of Theorem \ref{T:main} maps $(17,16,14,10,7,4,2,1)$ to $(21,11^2,9,5^2,3^3)$.

\subsection{A local description of Sylvester's map}

We mention that Sylvester's map can be implemented via pair interchanges. To describe this, let $\lambda\in{\mathcal O}(n)$ have length $\ell$. Set $\nu^{(1)}=(\frac{\lambda_1+1}{2},\frac{\lambda_1-1}{2})$, and for each $j=2,\dots,\ell$, form a partition $\nu^{(j)}$ by prepending $(\frac{\lambda_j+1}{2},\frac{\lambda_j-1}{2})$ onto $\nu^{(j-1)}$ and performing pair interchanges of the form $(a,a-1\mid b,c)\mapsto(b+1,c+1\mid a-1,a-2)$ (with the exceptional case $(1,0\mid b,c)\mapsto(b+1,c)$). Then $\mu=\nu^{(\ell)}$.

From the example in \ref{SS:Sylvester}, we know that $(13^2,11,5^2,3,1^2)\mapsto(14,11,10,8,6,3)$ under Sylvester's map. Starting with $\nu^{(1)}=(7,6)$, the sequence of $\nu^{(j)}$'s is:
$$
\begin{aligned}
(7,6\mid7,6)\rightarrow(8,7\mid6,5)&=:\nu^{(2)}\\ 
(6,5\mid 8,7\mid6,5)\rightarrow(9,8\mid7,6\mid4,3)&=:\nu^{(3)}\\ 
(3,2\mid9,8\mid7,6\mid4,3)\rightarrow(\textcolor{darkred}{10,9}\mid\textcolor{darkred}{8,7}\mid\textcolor{darkred}{5},3)&=:\nu^{(4)}\\ 
(3,2\mid10,9\mid8,7\mid5,3)\rightarrow(\textcolor{darkred}{11,10}\mid\textcolor{darkred}{9,8}\mid\textcolor{darkred}{6},3)&=:\nu^{(5)}\\ 
(2,1\mid11,10\mid9,8\mid6,3)\rightarrow(\textcolor{darkred}{12,11}\mid\textcolor{darkred}{10},8\mid6,3)&=:\nu^{(6)}\\ 
(1,0\mid12,11\mid10,8\mid6,3)\rightarrow(\textcolor{darkred}{13},11\mid10,8\mid6,3)&=:\nu^{(7)}\\ 
(1,0\mid13,11\mid10,8\mid6,3)\rightarrow(\textcolor{darkred}{14},11\mid10,8\mid6,3)&=:\mu\\ 
\end{aligned}
$$

For a partition $\mu$ of $n$, $2\mu:=(2\mu_1,\dots,2\mu_\ell)$ is a partition of $2n$ into even parts. Now our bijection ${\mathcal O}(n)\rightarrow{\mathcal D}(n)$ maps the odd partitions of $2n$ whose multiplicities are all even even bijectively onto the distinct partitions of $2n$ whose parts are all even. Concerning this bijection we have:

\begin{Corollary}
Let $\lambda\in{\mathcal O}(n)$ have image $\mu\in{\mathcal D}(n)$ under Sylvester's bijection. Then $\lambda^2\in{\mathcal O}(2n)$ has image $2\mu\in{\mathcal D}(2n)$ under the bijection of Theorem \ref{T:main}.
\begin{proof}
In order to construct $\mu$, we start with $\nu^{(0)}=()$ and for $j\geq1$, we construct $\nu^{(j)}$ by performing pair exchanges on $(\frac{\lambda_i+1}{2},\frac{\lambda_i-1}{2}\mid\nu^{(j-1)}_1,\nu^{(j-1)}_2\mid\nu^{(j-1)}_3,\nu^{(j-1)}_4\mid\dots)$. Then $\mu=\nu^{(\ell)}$.

Consider the image of $\lambda^2=(\lambda_1^2,\dots,\lambda_\ell^2)$ under the bijection  of Theorem \ref{T:main}. We start with the empty partition $\gamma^{(0)}=()$, and for $j\geq1$, we construct $\gamma^{(j)}$ by performing pair exchanges on $(\lambda_j+1,\lambda_j-1\mid\gamma^{(j-1)}_1,\gamma^{(j-1)}_2\mid\gamma^{(j-1)}_3,\gamma^{(j-1)}_4\mid\dots)$. Comparing this with the previous paragraph, we see that $\gamma^{(\ell)}=2\mu$.
\end{proof}
\end{Corollary}

For example our bijection maps $(13^4,11^2,5^4,3^2,1^4)$ to $2\mu=(28,22,20,16,12,6)$.

\subsection{Spin regular partitions}

Recall from Section \ref{SS:notation} that $\ell_r(\lambda)$ is the number of odd runs in $\lambda\vdash n$. We define the odd--with-distinct-small partitions of $n$ as
$$
\mathcal{ODS}(n):=\{\lambda\in{\mathcal O}(n)\mid m_i(\lambda)\leq1,\mbox{ for $1\leq i\leq\ell(\lambda)+\ell_r(\lambda)-2$}\}.
$$
Recall from Section \ref{SS:double} that $\mathcal{D}_{\leq2}(n)$ is the set of distinct partitions of $n$ which are the doubles of the spin regular partitions $\mathcal{S}(n)$ of $n$.

\begin{Proposition}
The bijection in Theorem \ref{T:main} restricts to a bijection $\mathcal{ODS}(n)\xrightarrow{\cong}\mathcal D_{\leq2}(n)$. In particular the number of spin regular partitions of\/ $n$ equals the number of odd partitions $\lambda$ of\/ $n$ which have no repeated part less than $\ell(\lambda)+\ell_r(\lambda)-1$.
\begin{proof}
As $|\mathcal{D}_{\leq2}(n)|=|\mathcal{S}(n)|$, the last assertion follows from the first.

Let $\lambda\in{\mathcal O}(n)$ have image $\mu\in{\mathcal D}(n)$ under the bijection of Theorem \ref{T:main}. Let $m\geq1$ such that $\ell(\mu)\leq2m\leq\ell(\mu)+1$. Recall the notation used in the proof of Theorem \ref{T:main}. In particular $\lambda=\alpha^2\circ\beta$, where $\alpha$ has $t$ parts and $\beta$ has distinct parts. Also $\beta$ is mapped to $\gamma^{(0)}$ under the bijection of Proposition \ref{P:Bressoud}, and there are distinct partitions $\gamma^{(1)},\dots,\gamma^{(t)}=\mu$. So $\ell(\gamma^{(0)})=\ell(\beta)+\ell_r(\lambda)$, according to Proposition \ref{P:Bressoud}.

Suppose first that $\mu\in{\mathcal D}_{\le2}(n)$. If $\mu_{2m-1}=2$ and $\mu_{2m}=0$, then $\lambda$ has the pair of parts $(2m-1)^2$, and all other repeated parts in $\lambda$ are $\geq2m-1$. Moreover by induction $\ell(\gamma^{(i)})=\ell(\gamma^{(i-1)})+2$ for $i=1,\dots,t-1$ and $\ell(\mu)=\ell(\gamma^{(t-1)})+1$. So $2m-1=\ell(\mu)=\ell(\lambda)+\ell_r(\lambda)-1$ and thus $\lambda\in\mathcal{ODS}(n)$.

The other possibility is that $\mu_{2m-1}=1$ or $\mu_{2m}>0$. Then arguing as in the first case, $\ell(\mu)=\ell(\lambda)+\ell_r(\lambda)$ and $\lambda$ no repeated part less than $\ell(\mu)$. So again $\lambda\in\mathcal{ODS}(n)$. 

Conversely, suppose that $\lambda\in\mathcal{ODS}(n)$. Notice that $(\alpha_1,\dots,\alpha_j)^2\circ\beta\in\mathcal{ODS}(*)$, for each $j=0,1,\dots,t$. Also $\gamma^{(j)}$ is the image of $(\alpha_1,\dots,\alpha_j)^2\circ\beta$ under the bijection of Theorem \ref{T:main}. So by induction we may assume that
$$
\ell(\gamma^{(t-1)})=\ell(\beta)+\ell_r(\beta)+2t-2=\ell(\lambda)+\ell_r(\lambda)-2.
$$
Then our hypothesis on $\lambda$ implies that $\alpha_t>\ell(\gamma^{(t-1)})$. If $\alpha_t=\ell(\gamma^{(t-1)})+1$ then $\mu$ is got by adding $2$ to each part of 
$\gamma^{(t-1)}$ and then appending one part $2$. So $\ell(\mu)=\ell(\lambda)+\ell_r(\lambda)-1$ and $\mu\in\mathcal D_{\leq2}(n)$. If instead $\alpha_t>\ell(\gamma^{(t-1)})+1$ then there is some $j\geq0$ such that $\mu_i=\gamma^{(t-1)}_i+2$, for $i=1,\dots,2j$, and $\mu_{2j+1}=\alpha_t-2j+1,\mu_{2j+2}=\alpha_t-2j-1$ and $\mu_i=\gamma^{(t-1)}_{i-2}$, for $i\geq 2j+3$. Thus $\ell(\mu)=\ell(\lambda)+\ell_r(\lambda)$ and once again $\mu\in\mathcal D_{\leq2}(n)$. 
\end{proof}
\end{Proposition}

\section{Acknowledgement}
Igor Pak told me about the {ytabelau} package for drawing Young diagrams in LaTeX. C. Bessenrodt sent me some useful comments on an earlier draft.

\end{document}